\documentclass[11pt,a4paper]{amsart}

\usepackage{amsmath,amsthm,amssymb}

\usepackage{amsfonts}
\usepackage{amssymb}
\usepackage{amsmath}
\usepackage{amscd}
\usepackage{enumitem}
\usepackage{mathrsfs}
\usepackage{lmodern}

\usepackage{latexsym}
\usepackage{bm}
\usepackage{graphicx}
\usepackage{wrapfig}
\usepackage{fancybox}

\usepackage{pgf}
\usepackage{tikz}
\usetikzlibrary{cd}

\usepackage{graphicx}
\newtheorem{theorem}{Theorem}[section]
\newtheorem{lemma}[theorem]{Lemma}
\newtheorem{proposition}[theorem]{Proposition}
\newtheorem{corollary}[theorem]{Corollary}
\newtheorem{definition}[theorem]{Definition}

\newtheorem{remark}{Remark}
\def\Id{\mbox{Id}}

\def\E{\mathbf E}
\def\dist{\operatorname{d}}

\def\N{{\mathbb N}}
\def\fs{{W}}

\def\cm{H}

\def\X{\chi}
\def\vect{\operatorname{span}}
\def\PP{P} 

\def\esp#1{\mathbf E\left[#1\right]}

\def\P{\mathbf P}

\def\D{\mathbb D}

\def\dif{\text{ d}}
\def\R{\mathbb{R}}
\def\car{\mathbf{1}}
\def\dx{\dif x}
\def\dz{\dif z}
\def\dy{\dif y}
\def\du{\dif u}

\def\dnn{\dif \nu^\sharp_n}

\def\SP{\Sigma}
\newcommand{\dom}{\operatorname{dom}}

\def\Lip{\operatorname{Lip}}
\newcommand\pr[1]{{\mathbf P}\left[#1\right]}
\newcommand\suite[3]{\left(#1_{#2}:\,#2\geq #3\right)}
\newcommand\dive[1]{\nabla^*_{#1}}
\newcommand\Dive[2]{\dive{#1}\left(#2\right)}

\def\<{\left \langle}
\def\>{\right\rangle}

\newcommand{\Pois}{N} 

\def\eust#1{\textcolor{magenta}{#1}}

\begin{document}
\title{Stein's method for diffusive limits of queueing
  processes
}

\author{E. Besan\c con}
\address{LTCI, Telecom Paris, I.P. Paris, France} \email{eustache.besancon@mines-telecom.fr}
\email{coutin@math.univ-toulouse.fr}
\author{L. Decreusefond}
\address{LTCI, Telecom Paris, I.P. Paris, France} \email{laurent.decreusefond@mines-telecom.fr}
\author{P. Moyal}
\address{Université de Lorraine, France} \email{pascal.moyal@univ-lorraine.fr}


\maketitle

\begin{abstract}
  Donsker Theorem is perhaps the most famous invariance principle result for
  Markov processes. It states that when properly normalized, a random walk
  behaves asymptotically like a Brownian motion. This approach can be extended
  to general Markov processes whose driving parameters are taken to a limit,
  which can lead to insightful results in contexts like large distributed
  systems or queueing networks. The purpose of this paper is to assess the rate
  of convergence in these so-called diffusion approximations, in a queueing
  context. To this end, we extend the functional Stein method introduced for the
  Brownian approximation of Poisson processes, to two simple examples: the
  single-server queue and the infinite-server queue. By doing so, we complete
  the recent applications of Stein's method to queueing systems, with results
  concerning the whole trajectory of the considered process, rather than its
  stationary distribution.
\end{abstract}


\section{Introduction}
\label{sec:intro}

The Markovian analysis of queueing systems often leads to stochastic processes
with an intricate evolution, for which the classical approach, which for
instance requires the computation of the stationary distribution, is
intractable. To gain some insights on the behavior of the process, it is then
customary to push the parameters to their limit and analyze the limiting process
which hopefully will reveal the inner structure of the model under analysis.
Diffusion approximations, as they are called, have been and still are the
subject of numerous papers (see \cite{robert_stochastic_2003,whitt_book} and
references therein to get a glimpse of the very rich literature on the subject).
The most naive example which comes to mind is the convergence of a normalized
Poisson process to a Brownian motion $B$: Letting $N^{\lambda}$ be a Poisson
process of intensity $\lambda$, we have that
\begin{equation}\label{eq_questa_BDM:6}
  \tilde{N}^{\lambda}\left(=: t \mapsto \frac{1}{\sqrt{\lambda}}(N^{\lambda}(t)-\lambda t)\right) \xrightarrow[\lambda\to \infty]{\text{dist. in } \D_{T}} B,
\end{equation}
where the convergence holds in distribution on the Skorokhod space as $\lambda$
goes to infinity. As the convergence in distribution is induced by a metric over
the set of probability measures, Eqn.~\eqref{eq_questa_BDM:6} just says that the
distance between the distribution of $\tilde{N}^{\lambda}$ and the distribution
of $B$ over $\D_{T}$ tends to zero. The next step is to determine the rate at
which this limit holds. The first study addressing this issue was due to Barbour
in the 90's \cite{barbour_steins_1990}. Since then, no papers on this subject
appeared until
\cite{coutin_steins_2012,kasprzak_diffusion_2017,shih_steins_2011}. These four
papers share the same common ground, relying on the so-called Stein method (SM)
\cite{stein_bound_1972}, see Section~\ref{sec:nutshell} for a modern
introduction. It is based on the fact that the topology of convergence in
distribution over a separable metric space $\X$ can be defined through the
distance
\begin{equation*}
  \dist(\mu,\nu)=\sup_{f\in \Lip-1} \left( \int_{\X} f\dif \mu-\int_{\X} f\dif \nu \right),
\end{equation*}
where $\Lip-1$ is the set of Lipschitz continuous function: $f\, :\,\X\to\R$
such that
\begin{equation*}
  |f(x)-f(y)|\le \dist_{\X}(x,y), \ \forall x,y\in \X.
\end{equation*}
An important avenue of literature has been dedicated to Stein's method in the
case $\X=\R$. Close to the class of models we have in mind, let us mention the
fruitful recent applications of the SM, to assess the rate of convergence of the
stationary distributions of various processes involved in queueing: Erlang-A and
Erlang-C systems in \cite{Braverman2016}; a system with reneging and phase-type
service time distributions (in which case the target distribution is the
stationary distribution of a piecewise Ornstein-Uhlenbeck process) in
\cite{Braverman2017}, single-server queues in heavy-traffic in \cite{GW19}.

When $\X$ is no longer $\R$, however, the development of the SM is much more
involved.
The main contribution of this work is to present applications of the Stein's
method to estimate the rate of convergence in functional CLT's arising in
queueing. Specifically, we complete existing functional Central Limit Theorems
of classical queueing systems (namely the M/M/1 queue and the 'pure delay'
M/M/$\infty$ system) by assessing the rate of convergence to the diffusion
limit, using Stein's method at the level of the whole stochastic process. These
two examples thus provide good illustrations of how the SM can be fruitfully
applied in a queueing context, at the process level. By completing two classical
asymptotic results with a simple rate of convergence estimate, for classes of
functions that have a practical meaning, the present work can thus constitute a
promising starting point for similar development regarding a larger class of
queueing systems.

This paper is organized as follows. In Section~\ref{sec:results}, we state our
main results for the diffusion approximation of the M/M/1 and M/M/$\infty$
queues. In Section~\ref{sec:interpol}, we introduce the intermediate processes,
i.e. the affine interpolation of both the Markov process under study and the
limit Brownian motion. Then, we estimate the error done by replacing the
original processes by their affine interpolations. Section~\ref{sec:interpol} is
devoted to the functional Stein method with which we control the distance
between the distributions of the interpolations defined above. The specific
calculations for the M/M/1 queue are done in Section~\ref{sec:proofMM1} and in
Section~\ref{sec:proofMMinfty} for the M/M/$\infty$ system. The Appendix
contains the proofs of two technical lemmas.
\section{The results}
\label{sec:results}
In this section we present our main results. In Theorems \ref{thm:mainMM1} and
\ref{thm:mainMMinfty} below, we provide bounds for the speed of convergence in
the diffusion approximation of two standard queueing systems: the single-server
and the infinite server queues, respectively. In what follows, for any $T>0$,
$\D:=\D_T$ denotes the Skorokhod space of càdlàg functions from $[0,T]$ to $\R$.
(We omit the dependance in $T$ for notational simplicity.) The
functional space $\Sigma$, to be properly defined in Definition \ref{def:Sigma}
below, is a subspace of the space of 1-Lipschitz continuous from $\D$ to $\R$. We denote for any
$U,V$ in $\D$,
\begin{align}
  \dist_{\Sigma}\left(U,V\right) &=\sup_{F\in \Sigma} \left|\esp{F\left(U\right)}-\esp{F\left(V\right)}\right|\label{eq:defdistSP}.
\end{align}
The distance $\dist_\Sigma$ is then the appropriate tool to introduce the
results to come.

\subsection{The M/M/1 queue}
\label{sec:MM1}
We first consider the classical M$_\lambda$/M$_\mu$/$1$ queue, that is, a single
server with infinite buffer, where the arrival is Poisson of intensity
$\lambda$, and the service times are i.i.d. from the exponential law
$\varepsilon(\mu)$. For all $t\ge 0$, we let $L^\dag(t)$ denote the number of
customers in the system (including the one in service, if any) at time $t$. The
process $L^\dag$ is clearly birth and death, and is ergodic if and only if
$\lambda/\mu<1$. If the initial size of the system is $x \in \N$, then
$L^{\dag}$ obeys the SDE
\begin{equation}
  \label{eq:SDEMM1}
  L^{\dag}(t) =x+\Pois_{\lambda}(t)-\int_{0}^{t}\mathbf{1}_{\{L^{\dag}(s^{-})>0\}}\Pois_{\mu}(ds),\quad t\ge 0,
\end{equation}
for two independent Poisson processes $\Pois_{\lambda}$ and $\Pois_{\mu}$. This
process is rescaled by accelerating time by a factor~$n$, while multiplying the
initial value, and then dividing the number of customers in the system at any
time by the same factor. For all $n\in\N^*$, the resulting normalized process
$\overline{L^\dag}$ then satisfies
\begin{align*}
  \overline {L^{\dag}_{n}}(t)                                &=x+\frac{\Pois_{n\lambda
                                 }(t)}{n}-\frac{\Pois_{n\mu}(t)}{n}+\frac{1}{n}\int_{0}^{t}\mathbf{1}_{\{L^{\dag}_{n}(ns^{-})=0\}}\dif \Pois_{n\mu
                                 }(s),\quad t\ge 0.
\end{align*}
It is a well established fact (see e.g. Proposition 5.16 in
\cite{robert_stochastic_2003}) that the sequence $\left(
  \overline{L^{\dag}_{n}}:\, n\ge 1\right)$ converges in probability and
uniformly over compact sets, to the deterministic function
\[\overline{L^{\dag}}\, :\,t \longmapsto \left(x+\lambda
    t - \mu t\right)^+,\] and that the process
\begin{equation*}
  Z^{\dag}_n\,:\,  t \longmapsto \frac{\sqrt{n}}{\sqrt{\lambda+\mu}}\ \left(\overline{L^{\dag}_{n}}(t)-\overline{L^{\dag}}(t)\right)
\end{equation*}
converges in distribution in $\D$ to the standard Brownian motion.

We can control the speed of the latter convergence. For that purpose, we bound
for any fixed $n$ and any horizon $T$, the $\SP$-distance between these
processes, defined by (\ref{eq:defdistSP}). We have the following result,
\begin{theorem}
  \label{thm:mainMM1}
  Suppose that $\lambda < \mu$ and let $T \le \frac{x}{\mu-\lambda}$. Then,
  there exists a constant $c_T$ such that for all $n\in\N$,
  \[\dist_{\SP}\left(Z^{\dag}_n,B\right) \leq {c_T \,\log\,n\over \log\log n\
      \sqrt{n}},\] where $B$ is a standard Brownian motion.
\end{theorem}

\noindent
The proof of Theorem \ref{thm:mainMM1} is deferred to Section
\ref{sec:proofMM1}.

\subsection{The infinite server queue M/M/$\infty$}
\label{subsec:MMinfty}

We now turn to the classical ``infinite server'' M$_\lambda$/M$_\mu$/$\infty$
queue: a potentially unlimited number of servers attend customers that enter the
system following a Poisson process of intensity $\lambda$, requesting service
times that are exponentially distributed of parameter $\mu$ (where $\lambda,\mu
>0$).

Assuming throughout that the system is initially empty, let $L^{\sharp}(t)$
denote the number of customers in the system at time $t$. The process
$L^{\sharp}$ is a.s. an element of $\D$; this is an ergodic Markov process which
obeys the SDE
$$L^{\sharp}(t)=N_{\lambda}(t)-\sum_{i=1}^{+\infty}\int_{0}^{t}\mathbf{1}_{\{L^{\sharp}(s^{-})\geq i\}}{N}^{i}_{\mu}(ds),\quad t\ge 0,$$
where $N_\lambda$ is a Poisson process of intensity $\lambda$ and the
$N_\mu^i$'s are independent Poisson processes of intensity $\mu$.
The classical scaling of the process $L^{\sharp}$ goes as follows; we accelerate
time by a factor $n\in \N^*$ and divide the size of the system by $n$. The
corresponding $n$-th rescaled process is then defined by
$$\overline{L^{\sharp}_n}: t \longmapsto \frac{{N}_{\lambda n}(t)}{n}-\frac{1}{n}\sum_{i=1}^{+\infty}\int_{0}^{t}\mathbf{1}_{\{\overline{L^{\sharp}_n}(s^{-})>\frac{i}{n}\}}{N}^{i}_{\mu}(ds).$$
It is a well known fact (see e.g. Theorem 6.13 in \cite{robert_stochastic_2003})
that the sequence of processes $\left(\overline{L^{\sharp}_n},n\geq 0\right)$
converges in $L^1$ and uniformly over compact sets to the deterministic function
\begin{equation}
  \label{eq:limfluinfty}
  \overline{L^{\sharp}}\, :\, t \longmapsto \rho -\rho e^{-\mu t},
\end{equation}
where $\rho =\lambda/\mu$. Moreover, if we define for all $n$ the process
\begin{equation}
  \label{eq:defGN}
  Z^{\sharp}_{n}\, :\, t \longmapsto \sqrt{n}\left(\overline{L^{\sharp}_n}(t)-\overline{L^{\sharp}}(t)\right),
\end{equation}
then the sequence $\suite{Z^{\sharp}}{n}{0}$ converges in distribution to the
process $Z^{\sharp}$ defined by
\begin{equation}
  Z^{\sharp}\, :\,  t \longmapsto Z^{\sharp}(t)=Z^{\sharp}(0)e^{-\mu t}+\int_{0}^{^t}e^{-\mu(t-s)}\sqrt{h(s)}\,\dif B(s),
  \label{ProcessusX}
\end{equation}
where $h(t)=\lambda\left(2 - e^{-\mu t}\right)$ for all $t\geq 0$; see e.g.
\cite{borovkov_limit_1967} or Theorem 6.14 in \cite{robert_stochastic_2003}. We
have the following result,
\begin{theorem}
  \label{thm:mainMMinfty}
  For any $T>0$, there exists a constant $c_T>0$ such that for all $n \ge 1$,
  \[\dist_{\SP}(Z^{\sharp}_n,Z^{\sharp})\leq\frac{c_T\log n}{\log \log n\,
      \sqrt{n}}\cdotp\]
\end{theorem}
\noindent We defer the proof of Theorem \ref{thm:mainMMinfty} to Section
\ref{sec:proofMMinfty}.

\subsection{Consequences}
\label{sec:consequences}
Let us quote a few
functionals which are often encountered in queueing analysis, and which are
regular enough to be elements of $\Sigma$ (see Definition \ref{def:Sigma}
below). This is the case, first, for the function $F_f$, that is defined for any
mild enough function $f$ and $T>0$, by
\[F_f:
  \begin{cases}
    \D &\longrightarrow \R\\
    x=\Bigl(x_t,\, t\in [0,T]\Bigr) &\longmapsto \dfrac{1}{T} \displaystyle\int_0^T f(x_s) \dif s,
  \end{cases}\] observing that $F_f\left(X\right)$ goes to $\mathbb
E_\pi\left[f\right]$ for large $T$ whenever the Markov process $X$ is ergodic of invariant
probability $\pi$. The proof is deferred to
Remark~\ref{rem:1} below. Similarly, for $M\ge 0$ and $p\ge 2$,
\[F_{M,p}:
  \begin{cases}
    \D &\longrightarrow \R\\
    x &\longmapsto \left(\displaystyle\int_0^T |x_s\wedge M|^{p } \dif s\right)^{1/p}.
  \end{cases}\]
also belongs
to the set of admissible test functions. Observe that for $M$ and $p$ large enough,
$F_{M,p}(x)$ can be considered as an ersatz to  $\sup_{s\le T}|x_{s}|$.

For
any of these functionals $F$, if $\dist (\P_{X_{}n},\P_{X})$ tends to $0$ as
$n^{-\alpha}$, then the distribution of the random variables $(F(X_{n}),\, n\ge
1)$ converges in the sense of a damped Kantorovith-Rubinstein distance at a rate $n^{-\alpha}$:
\begin{equation*}
  \sup_{\varphi\in \mathcal C_{b}^{3}}\left| \esp{\varphi\Bigl( F(X_{n}) \Bigr)} -\esp{\varphi\Bigl( F(X) \Bigr)}\right|\le c \, n^{-\alpha},
\end{equation*}
where $\mathcal C_{b}^{3}$ is the set of three times differentiable functions
from $\R$ to $\R$ with bounded derivatives of any order.
Note that this kind of result is inaccessible via the standard Stein's method in
dimension~1, since we usually cannot achieve the first step of the SM, which
consists in devising a functional characterization of the distribution of
$F(X)$.

\section{Interpolation of Markov processes}
\label{sec:interpol}
To prove Theorems \ref{thm:mainMM1} and \ref{thm:mainMMinfty}, we will be led to
bound the distance between the affine interpolation of the Markov process under
consideration ($Z^\dag$ in the first case, $Z^\sharp$ in the second), and that
of a (time-changed) Brownian motion, on a finite horizon $T>0$.

For fixed $T>0$ and $n\in\N^*$, let us denote throughout this paper, by $t^n_i$,
$i=0,...,n$, the points of the discretization of $[0,T]$ of constant mesh $T/n$,
namely $t^n_i = iT/n$ for all $i=1,...,n$. For a function $f\in\D$, denote by
$\Pi_nf$ its affine interpolation on the latter grid, that is, for all $t \in
[0,T]$, for $k=1,...,n$ such that $t\in\left[t_{k-1}^n, t^n_k\right]$,
\begin{equation}
\label{eq:defPin}
\Pi_nf(t)= {n\over T}\left(f\left(t^n_k\right) -
    f\left(t^n_{k-1}\right)\right)\left(t-t^n_{k-1}\right)+f\left(t^n_{k-1}\right).
    \end{equation}
An immediate computation then shows that for all $t \le T$ and for $k$ as above,
we have that
\begin{align}
  \Pi_nf(t) &= {n \over T}\Biggl(\left(f\left(t^n_k\right) - f\left(t^n_{k-1}\right)\right) \left(t-t^n_{k-1}\right)+\sum_{i=1}^{k-1}\left(f\left(t^n_{i}\right)-f\left(t^n_{i-1}\right) \right) {T\over n} \Biggl)+ f(0)\nonumber\\
            &={n \over T}\left(\sum_{i=1}^n\left(f\left(t^n_i\right) - f\left(t^n_{i-1}\right)\right) \int_0^t \mathbf 1_{\left[t^n_{i-1},t^n_i\right)}(s)\,ds\right)+f(0)\nonumber\\
            &=\sqrt{{n \over T}}\left(\sum_{i=1}^n\left(f\left(t^n_i\right) - f\left(t^n_{i-1}\right)\right) h^n_i(t)\right)+f(0),\label{eq:interpol}
\end{align}
where
\begin{equation}
  \label{eq:defhni}
  h^n_i : t \longmapsto \sqrt{{n \over T}}\int_0^t \mathbf 1_{\left[t^n_{i-1},t^n_i\right)}(s)\,ds,\quad i=1,...,n.
\end{equation}
In what follows, $B$ denotes a standard one dimensional Brownian motion and observe that $\Pi_nB$ and $B_n$ defined by (\ref{eq:defBn}) below, are equal in law.
Let us define the space
\begin{equation}
\label{eq:defW}
W := W_T = \left\{\mbox{continuous mappings from $[0,T]$ to $\R$}\right\},
\end{equation}
which, furnished with the sup norm $\parallel . \parallel_{W}$ defined for all $f\in W$ by $\parallel f \parallel = \sup_{x\in[0,T]} |f(x)|$, is a Banach space.

The Proposition 13.20 in \cite{friz_multidimensional_2010}
states that for all $T>0$, for some $c>0$, we have that
\begin{equation}
  \esp{\parallel \Pi_{n}B - B\parallel_{W}}\le c\, n^{-1/2},\,\text{ for all } n\in\N^*.
  \label{dist B, Bgamma}
\end{equation}

We now estimate the distance between the sample-paths of Birth-and-Death
processes and their interpolation. Specifically,

\begin{lemma}
  \label{thm:BD}
  Let $T>0$, $n\in\N^*$, and let $X$ be a $\N$-valued Markov jump process on
  $[0,T]$ of infinitesimal generator $\mathscr A$. Suppose that there exist two
  constants $J \in \N$ and $\alpha>0$ such that
  \begin{itemize}
  \item the magnitude of the jumps of $X$ is bounded by $J>0$, i.e. for all $i,j
    \in \N$, $\mathscr A(i,j)=0$ whenever $|j-i| > J$;
  \item the intensities of the jumps of $X$ are bounded by $n\alpha$, i.e. for
    all $i,j \in \N$, $i \ne j$, $\mathscr A(i,j) \le n\alpha$.
  \end{itemize}
  Then,
  \[\esp{\parallel X-\Pi_nX \parallel_{W}} \le 2J\ {\log\,n \over
      \log\log n}\cdotp\]
\end{lemma}

\begin{proof}
  Fix $n\in \N$ and within this proof, set for $t^n_i={iT \over n}$ for
  $i=0,...,n$.
  For any $t\in [0,T]$, for $i\le n$ such that $t \in
  \left[t^n_{i-1},t^n_{i}\right]$ we have that
  \begin{align*}
    \left|X(t)-\Pi_nX(t)\right|
    &= \left|X(t) - X\left(t^n_{i-1}\right)-{n \over T}\left(t-t^n_{i-1}\right)\left(X\left(t^n_{i}\right)- X\left(t^n_{i-1}\right)\right)\right|\\
    &\leq 2 \sup_{t\in \left[t^n_{i-1},t^n_{i}\right]}\left|
      X(t)- X\left(t^n_{i-1}\right)\right|,
  \end{align*}
  so that
  \begin{equation}
    \esp{\parallel X -\Pi_nX \parallel_{W}}
    \le 2\esp{\max_{i\in [0,n-1]}\sup_{t\in \left[t^n_{i-1};t^n_{i}\right]}\left| X(t)-X\left(t^n_{i-1}\right)\right|}.\label{eq:BD1}
  \end{equation}
  But for any $i$ and any $t \in \left[t^n_{i-1},t^n_{i}\right]$, we have that
  \[\left|X(t)-X\left(t^n_{i-1}\right)\right| \le J\left(A_n^i +
      D_n^i\right),\] where $A_n^i$ and $D_n^i$ denote respectively the number
  of up and down jumps of the process $X$ within the interval
  $\left[t^n_{i-1},t^n_{i}\right]$. In turn, by assumption $A_n^i+D_n^i$ is
  stochastically dominated by a Poisson r.v., say $P^i$, of parameter $\alpha n
  {T \over n} = \alpha T$. All in all, we obtain with (\ref{eq:BD1}) that
  \[\esp{\parallel X-\Pi_nX\parallel_{W}} \le 2J \,\esp{\max_{i\in
        [1,n]}P^i},\] and we conclude using Proposition
  \ref{prop:momentPoisson}. \hfill\qed \end{proof}

\section{A functional Stein method}
\label{sec:thm1}


\subsection{Stein's method in a nutshell}
\label{sec:nutshell}
Say that we want to compare a distribution $\nu$ on $\R^{n}$, $q\ge 1$, to the standard
Gaussian distribution on $\R^{n}$, denoted by $\mu_{n}$. Consider the processes
\begin{equation}\label{eq_fluid_queueing_BDM2:5}
  t \mapsto X(x,t)=e^{-t}x+\sqrt{2}\int_{0}^{t}e^{-(t-s)}\dif B^n(s),\quad x\in \R^n,
\end{equation}
where $B^n$ is an ordinary Brownian motion in $\R^n$. For all $x$, it is a Gaussian process
whose distribution at time $t$ is a Gaussian law of mean $e^{-t}x$ and
covariance matrix $(1-e^{-2t})\Id_{n}$. For $t\ge 0,\,x\in\R^n$, let
\begin{equation*}
  Q^{n}_{t}f(x)=\esp{f(X(x,t))}=\int_{\R^{n}}f(e^{-t}x+\beta_{t}y)\dif \mu_{n}(y),
\end{equation*}
where $\beta_{t}=\sqrt{1-e^{-2t}}$. The dominated convergence theorem entails
that
\begin{equation*}
  Q^{n}_{t}f(x)\xrightarrow{t\to \infty} \int_{\R^{n}}f\dif \mu_{n},\quad x\in\R^n.
\end{equation*}
Moreover, the Dynkin Lemma and the Itô formula entail (see \cite{ethier86}) that
\begin{equation}\label{eq_questa_BDM:1}
  Q^{n}_{t}f(x)-f(x)=\int_{0}^{t} A^{n}Q^{n}_{s}f(x)\dif s,\quad x\in\R^n,\,t\ge 0,
\end{equation}
where for $f $ regular enough
\begin{equation*}
  A^{n}f(x)=: \left.\frac{d}{dt}(Q_{t}^{n}f)(x)\right|_{t=0}=  \<x,\, d_{n}
  f(x)\>_{\R^{n}}-\Delta_{n} f(x).
\end{equation*}
The notation $d_{n}f$ represents the usual gradient of $f\, :\,\R^{n}\to \R$ and
$\Delta_{n} f$ is its Laplacian. Integrate both sides of \eqref{eq_questa_BDM:1}
with respect to $\nu$ to obtain the so-called Stein-Dirichlet representation:
for any $f$ in a well chosen functional space $\mathcal F$ (i.e. we must at
least  require
that the previous limits do exist and that $A^{n}Q^{n}f$ is well defined and integrable
for $f\in \mathcal F$),
\begin{equation}\label{eq_fluid_queueing_BDM2:6}
  \dist_{\mathcal F} \bigl(\nu,\, \mu_{n}\bigr)=\sup_{f\in \mathcal F} \int_{\R^{n}} \int_{0}^{\infty} A^{n}Q^{n}_{s}f(x)\dif s \dif \nu(x).
\end{equation}
This formula is the first step of the modern approach to the Stein's method, see
\cite{decreusefond_stein-dirichlet-malliavin_2015}.

\subsection{Generalization to infinite dimension}
As we mentioned above, the proofs of Theorems \ref{thm:mainMM1} and
\ref{thm:mainMMinfty} critically rely on bounding the distance between the
affine interpolations of the Markov processes under consideration and their
diffusion approximations. For this, we need to go to a functional setup, that
is, to bound a similar expression to (\ref{eq_fluid_queueing_BDM2:6}) when the
target measure is that of a Gaussian process, instead of a $d$-dimensional
Gaussian random variable. This is done in the main result of this section,
Theorem \ref{thm:stein_final}.

Fix $T>0$ and an integer $n\ge 1$. Recall (\ref{eq:defhni}), and define
the following subspace of $W$,
\begin{equation*}
  W_{n}=\vect\{h_{j}^{n},j=1,\cdots, n\},
\end{equation*}
equipped with the sup-norm $\parallel . \parallel_W$. Now define the process
\begin{equation}
\label{eq:defBn}
  B_{n}=\sum_{j=1}^{n} Y_{j}\, h_{j}^{n},
\end{equation}
 where $(Y_{j},\, j=1,\cdots,n)$ is a Gaussian vector of distribution $\mu_{n}$.
Clearly, $B_{n}$ belongs to $W_{n}$ with probability~$1$, thereby defining a
Gaussian distribution, denoted by $\pi_{n}$, on $W_{n}$. We also need a space to
define the gradients. For this, we now consider the space
\begin{equation*}
  H_{n}=\vect\{h_{j}^{n},j=1,\cdots, n\}
\end{equation*}
equipped with the scalar product
\begin{equation*}
  \<h,\, g\>_{H_{n}}=\int_{0}^{T}h'(s)g'(s)\dif s,\,h,g,\in H_n.
\end{equation*}
\begin{remark}
  Distinguishing between the spaces $H_{n}$ and $W_{n}$ may seem spurious, as these are
  algebraically the same set, and only differ by their norms. Actually, $W_{n}$
  (respectively, $H_{n}$) is
  the image by the map $\Pi_{n}$ defined by (\ref{eq:defPin}), of the set $W$ defined by (\ref{eq:defW})
  (resp., of the Banach set $H$ - dense in $W$ - that is defined by \eqref{eq_questa_BDM:7} below). 
  An intuitive explanation of our need to introduce the space $H$ is as follows:
  As mentioned above, the control of the properties of the solution of the Stein
  equation requires dealing with the derivative of this function.
  In functional
  spaces, the usual notion of derivative is replaced by that of Fréchet
  differential: A function $F$ from $W$ into $\R$ is Fréchet differentiable whenever for any $w,w'\in W$, the
  function
  \begin{equation*}
    \varepsilon \longmapsto F(w+\varepsilon \, w')
  \end{equation*}
  is differentiable with respect to $\varepsilon$ in a neighbor of $0$. For
  technical reasons, which are detailed in \cite{coutin_rough}, assuming that $F$
  is Fréchet differentiable in a probabilistic context is too stringent a
  condition. It turns out that the notion of weak differentiability, i.e.
  the
  function
  \begin{equation*}
    \varepsilon \longmapsto F(w+\varepsilon \, h)
  \end{equation*}
  is differentiable with respect to $\varepsilon$ in a neighbor of $0$ for any
  $w\in W$ and $h\in H$ is sufficient for what we aim to do, and do not put too
  strong a constraint on~$F$. Hence the necessity of considering $W$ (the space
  into which the sample-paths of our processes take place) and $H$ (the set of the
  admissible directions of differentiation), and thus to distinguish between the spaces $W_{n}$ and $H_{n}$ at the
  level of the interpolated processes.
\end{remark}

The space $H_{n}^{\otimes (2)}$ is then the vector space
\begin{equation*}
  H_{n}^{\otimes (2)} = \vect\left\{h_{j}^{n}\otimes h_{k}^{n}=\Bigl( (s_{1},s_{2})\longmapsto h_{j}^{n}(s_{1})h_{k}^{n}(s_{2}) \Bigr),\, j,k=1,\cdots,n\right\},
\end{equation*}
equipped with the scalar product: For any $h,g\in H_{n}^{\otimes (2)}$,
\begin{equation*}
  \<h,g\>_{H_{n}^{\otimes (2)}}=\int_{0}^{T}\int_{0}^{T}\frac{\partial^{2} h}{\partial s_{1}\partial s_{2}}(s_{1},s_{2})\ \frac{\partial^{2} g}{\partial s_{1}\partial s_{2}}(s_{1},s_{2})\dif s_{1}\dif s_{2} .
\end{equation*}
For a regular enough function $f\, :\, W_{n}\to \R$, we denote by $D_{n}f$ its
differential, i.e. for any $w\in W_{n}$, for any $h\in H_{n}$,
\begin{equation}
  \label{eq_questa_BDM:3bis}
  \<D_{n}f(w),\, h\>_{H_{n}}=\left.\frac{d}{d\varepsilon} f(w+\varepsilon h)\right|_{\varepsilon=0}.
\end{equation}
We even need to iterate this definition and consider the second order
differential, for any $w\in W_{n}$, for any $h_{1},h_{2}\in H_{n}$,
\begin{equation}\label{eq_questa_BDM:3}
  \<D^{(2)}_{n}f(w),\, (h_{1},h_{2})\>_{H_{n}^{\otimes (2)}}=\left.\frac{\partial^{2}}{\partial\varepsilon_{1}\partial\varepsilon_{2}} f(w+\varepsilon_{1} h_{1}+\varepsilon_{2}h_{2})\right|_{\varepsilon_{1}=\varepsilon_{2}=0}.
\end{equation}
The map
\[T_{n}\, :\, \left\{\begin{array}{ll}
                       \R ^{n}&\longrightarrow W_{n}\\
                       (y_{1},\cdots, y_{n})&\longmapsto
                                              \displaystyle\sum_{j=1}^{n}
                                              y_{j}\, h_{j}^{n},
                     \end{array}\right.\]
                 is a morphism of probability spaces, i.e. it is linear,
                 continuous and preserves the probability measures: the image
                 measure of $\mu_{n}$ by $T_{n}$ is actually $\pi_{n}$. Then, we
                 can generalize the construction we just followed on $\R^{n}$ to
                 the finite dimensional space $W_{n}$. The family of maps
                 $(P_{t}^{n},\, t\ge 0)$ is defined as follows: $P_{0}^{n}=\Id$
                 and for all $t>0$,
                 \[P_t^n:\left\{\begin{array}{ll}
                                  L^{1}(\pi_{n}) & \longrightarrow \,L^{1}(\pi_{n})\\
                                  f & \longmapsto \left(w\longmapsto
                                      P_{t}^{n}f(w)=\displaystyle\int_{W_{n}}f(e^{-t}w+\beta_{t}\zeta)\dif
                                      \pi_{n}(\zeta)\right)
                                \end{array}\right.\]
                            Since $T_{n}$ is linear and since $\pi_{n}$ is the
                            image of $\mu_{n}$ by $T_{n}$, we easily see that
                            for any $t\ge 0$
                            \begin{equation*}
                              Q^{n}_{t}(f\circ T_{n})=\bigl( P_{t}^{n}f \bigr)\circ T_{n},
                            \end{equation*}
                            which can be written
                            \begin{equation}
                              \label{eq_questa_BDM:5}
                              P_{t}^{n}f= Q^{n}_{t}(f\circ T_{n})\circ T_{n}^{-1}.
                            \end{equation}
                            Thus, $(P_{t}^{n},\, t\ge 0)$ is a semi-group such
                            that
                            \begin{equation*}
                              P_{t}^{n}f(w)\xrightarrow{t\to \infty} \int_{W_{n}}f\dif \pi_{n},\quad w\in W_{n}.
                            \end{equation*}
                            From \eqref{eq_questa_BDM:5}, we also infer that for
                            $f\, :\, W_{n}\to \R$ twice differentiable,
                            \begin{equation*}
                              L^{n}f=: \left.\frac{d}{dt}(P_{t}^{n}f)\right|_{t=0}=A^{n}(f\circ T_{n})\circ T_{n}^{-1}.
                            \end{equation*}
                            Hence, we have that
                            \begin{equation}\label{eq_questa_BDM:2}
                              P_{t}^{n}f(w)-f^{n}(w)=\int_{0}^{t} L^{n}P^{n}_{s}f(w)\dif s,\quad w\in W_{n},\,t\ge 0,
                            \end{equation}
                            where for $f$ regular enough
                            \begin{equation*}
                              L^{n}f(w)=  \< D_{n} f(w),\, w\>_{H_{n}}-\sum_{j=1}^{n} \< D^{(2)}f(w),\,h_{j}^{n}\otimes h_{j}^{n} \>_{H_{n}^{\otimes (2)}}.
                            \end{equation*}
                            Thus, for any measure $\nu_{n}$ on $W_{n}$,
                            \begin{equation}\label{eq_fluid_queueing_BDM2:61}
                              \dist_{\mathcal F_{n}} \bigl(\nu_{n},\, \pi_{n}\bigr)=\sup_{f\in \mathcal F_{n}} \int_{W_{n}} \int_{0}^{\infty} L^{n}P^{n}_{s}f(w)\dif s \dif \nu_{n}(w),
                            \end{equation}
                            where $\mathcal F_{n}$ is a space of regular enough
                            test functions from $W_{n}$ into $\R$. We can now
                            precise which kind of test functions we are going to
                            consider. In view
                            of~\eqref{eq_fluid_queueing_BDM2:61}, it must
                            contains twice differentiable functions but for
                            technical reasons, we need more than that.
                            \begin{definition}
                              A function $f\, :\, W_{n}\to \R$ is said to belong
                              to the class $\Sigma_{n}$ whenever it is
                              $1$-Lipschitz continuous, twice differentiable in
                              the sense of \eqref{eq_questa_BDM:3}, and we have
                              \begin{equation}
                                \label{eq_preliminaries:5}
                                \sup_{w\in W_{n}}   \left|     \<D_{n}^{(2)} f_{n}(w)-D_{n}^{(2)} f_{n}(w+g),\, h\otimes k \>_{H_{n}^{\otimes (2)}}\right|\le \|g\|_{\fs}\, \|h\|_{L^{2}}\|k\|_{L^{2}},
                              \end{equation}
                              for any $g\in \fs_{n},\ h,k \in
                              H_{n}$. 
                              \label{DefSigman}
                            \end{definition}
                            Actually, in the definition of the distance between
                            distributions of processes, the test functions are
                            defined on the whole space $W$. Hence, we must find
                            a class of functions whose restriction to $W_{n}$
                            belong to $\Sigma_{n}$ for any $n\ge 1$. This
                            involves the notion of $H$-differential on~$W$. Let
                            \begin{equation}\label{eq_questa_BDM:7}
                              H=\left\{ h,\, \exists ! h'\in L^{2}([0,T]) \text{ such that } h(t)=\int_{0}^{t } h'(s)\dif s \right\}.
                            \end{equation}
                            It is an Hilbert space when equipped with the scalar
                            product
                            \begin{equation*}
                              \<h,g\>_{H}=\int_{0}^{T}h'(s) g'(s)\dif s.
                            \end{equation*}
                            A function $f\, :\, W\to \R$ is said to be twice
                            $H$-differentiable whenever for any $w\in W$, for
                            any $h\in H$, the
                            function
                            \[\left\{ \begin{array}{ll}
                                        \R &\longrightarrow \R\\
                                        \varepsilon &\longmapsto f(w+\varepsilon h)
                                      \end{array}\right.\]
                                  is twice differentiable in a neighbor of~$0$. We denote by
                                  $D f$ and $D^{(2)}f$ its first and second
                                  order gradient, defined by
                                  \begin{align*}
                                    \< D f(x),\, h\>_{H}&=\left.\frac{d}{d\varepsilon} f(x+\varepsilon h)\right|_{\varepsilon=0},\\
                                    \< D^{(2)} f(x),\, h_{1}\otimes h_{2}\>_{H^{\otimes (2)}} &=\left.\frac{\partial^{2}}{\partial\varepsilon_{1}\partial \varepsilon_{2}} f(w+\varepsilon_{1} h_{1}+\varepsilon_{2}h_{2})\right|_{\varepsilon_{1}=\varepsilon_{2}=0}.\\
                                  \end{align*}
                                  \begin{definition}
                                    \label{def:Sigma}
                                    The class $\Sigma$ is the set of
                                    $1$-Lipschitz continuous, twice
                                    $H$-differentiable functions such that
                                    \begin{equation*}
                                      \sup_{w\in W}   \left|     \<D^{(2)} f(w)-D^{(2)} f(w+g),\, h\otimes k \>_{H^{\otimes (2)}} \right|\le \|g\|_{\fs}\, \|h\|_{L^{2}}\|k\|_{L^{2}},
                                    \end{equation*}
                                    for any $g\in \fs,\ h,k \in H$.
                                  \end{definition}
                                  For $f\,:\,W\to \R$, let $f_{n}=f_{|W_{n}}$.
                                  If $f $ is once $H$-differentiable, then, we
                                  have that for any $w_{n}\in W_{n}$, any $j\in
                                  \{0,\cdots,n-1\}$,
                                  \begin{multline}\label{eq_questa_BDM:4}
                                    \< Df(\mathfrak e(w_{n})),\,
                                    h_{n}^{j}\>_{H}=\left.
                                      \frac{d}{dt}f(\mathfrak
                                      e(w_{n}+\varepsilon h_{n}^{j
                                      }))\right|_{\varepsilon=0} =\left.
                                      \frac{d}{dt}f_{n}(w_{n}+\varepsilon
                                      h_{n}^{j })\right|_{\varepsilon=0}\\= \<
                                    D_{n}f_{n}(w_{n}),\, h_{n}^{j}\>_{H_{n}}.
                                  \end{multline}
                                  Thus, it is straightforward that if $f$
                                  belongs to $\Sigma$ then $f_{n}$ belongs to
                                  $\Sigma_{n}$ for any $n\ge 1$.
                                  \begin{remark}\label{rem:1} We can now show how to prove
                                    that the functionals mentioned in the
                                    introduction do belong to $\Sigma$. Consider
                                    the first one :
                                    \begin{equation*}
                                      F_{f}(x)=\frac{1}{T}\int_{0}^{T}f(x_{s})\dif s.
                                    \end{equation*}
                                    Then, for any $x,y\in W$,
                                    \begin{equation*}
                                      |F_{f}(x)-F_{f}(y)|\le \|x-y\|_{W}
                                    \end{equation*}
                                    provided that $f$ is Lipschitz continuous.
                                    Moreover, a classical computation shows that
                                    \begin{multline*}
                                      \<D^{(2)}F_{f}(x+g)-D^{(2)}f(x),\, h\otimes k\>_{H^{\otimes (2)}}\\
                                      =\dfrac{1}{T}\int_{0}^{T
                                      }\bigl(f''(x_{s}+g(s))-f''(x_{s})\bigr)h_{s}k_{s}\dif
                                      s.
                                    \end{multline*}
                                    Hence $F_{f}$ belongs to $\Sigma$ as long as
                                    $f''$ does exist and is Lipschitz
                                    continuous. The other cases are handled similarly.
                                  \end{remark}
                                  \subsection{Functionals of Poisson marked
                                    point processes}
                                  \label{sec:functionalStein}
                                  Let $N_\nu$ be a marked point process on
                                  $E=[0,T]\times \R^{+}$ whose jump times are
                                  denoted by $(T_{n},n\ge 1)$, and jumps
                                  magnitude by $(Z_{n},n\ge 1)$. It is said to
                                  be a Poisson marked point process of (diffuse)
                                  control measure~$\nu$ whenever for any
                                  function
                                  \begin{math}
                                    u=\left( u(s,z),\, s\in [0,T],\, z\in \R^{+}
                                    \right)
                                  \end{math}
                                  in $L^{2}(\nu)$,
                                  the process
                                  \begin{equation*}
                                    t \longmapsto  ( \dive{\nu} u)(t)= \sum_{T_{n}\le t } u(T_{n},Z_{n})-\int_{0}^{t}\int_{\R^{+}} u(s,z)\dif \nu(s,z)
                                  \end{equation*}
                                  is a square integrable martingale. We set
                                  \begin{equation}\label{eq:defdive}
                                    \dive{\nu} u= ( \dive{\nu} u)(T).
                                  \end{equation}
                                  Consider the so-called discrete gradient
                                  \cite{DecreusefondStochasticmodelinganalysis2012,privault_stochastic_2009},
                                  \begin{equation*}
                                    \nabla_{s,z}f(N_\nu)=f(N_\nu+\varepsilon_{s,z})-f(N_\nu),\,s\in[0,T],\,z\in\R^+,
                                  \end{equation*}
                                  where $N_\nu+\varepsilon_{s,z}$ represents the
                                  sample-path $N_\nu$ to which we add an atom at
                                  time $s$ of size $z$. Since $\nu$ is diffuse,
                                  there is a zero probability that an atom at
                                  time~$s$ already exists in $N_\nu$. Similarly,
                                  we denote by $N_\nu-\varepsilon_{s,z}$ the
                                  sample-path $N_{\nu}$ to which we remove the
                                  atom $\varepsilon_{s,z}$ provided it is
                                  present in $N_{\nu}$, otherwise $N_{\nu}$
                                  remains unchanged.
                                  \begin{definition}
                                    We define the domain of $\nabla$ as
                                    \begin{equation*}
                                      \dom \nabla=\left\{ f,\ \esp{\int_{[0,T]\times \R^{+}}|\nabla_{s,z}f(N_\nu)|^{2}\dif \nu(s,z)}<\infty, \right\}.
                                    \end{equation*}
                                  \end{definition}
                                  We then have
                                  the integration by parts formula
                                  \cite{DecreusefondStochasticmodelinganalysis2012}:
                                  \begin{lemma}
                                    \label{thm:ipp_Poisson}
                                    For $u\in L^{2}(\nu)$, for $f\in \dom \nabla$,                                    we have that
                                    \begin{equation}
                                      \label{eq_preliminaries:1}  \esp{ f(N_\nu)\  \dive{\nu} u}=\esp{ \int_{[0,T]\times \R^{+}}  \nabla_{s,z}f(N_\nu) \, u(s,z) \dif \nu(s,z) }.
                                    \end{equation}
                                  \end{lemma}
                                  For the sake of completeness, we reproduce the
                                  proof of this identity, which is a mere
                                  rewriting of the Campbell-Mecke formula for
                                  Poisson processes.
                                  \begin{proof}
                                      By the very definition of $\nabla$,
  \begin{multline}\label{eq:1}
    \esp{\int_{E}\nabla_{s,z}f(N_\nu) \, u(s,z)\dif \nu(s,\,z)}\\=\esp{\int_{E} f(N_\nu+
      \varepsilon_{s,z})u(s,z) \dif \nu(s,\,z)}
    -\esp{\int_{E} f(N_\nu)u(s,z) \dif \nu(s,\,z)}.
  \end{multline}
  The Campbell-Mecke formula for Poisson processes says that
  \begin{multline}\label{eq_poisson2:5}
    \esp{\int_{E} f(N_\nu+
      \varepsilon_{s,z})u(s,z) \dif \nu(s,\,z)}=\esp{ f(N_\nu)\sum_{T_{n}\le T } u(T_{n},Z_{n})}.
  \end{multline}
Plug \eqref{eq_poisson2:5} into the right-hand-side of \eqref{eq:1} to obtain~\eqref{eq_preliminaries:1}.
                                  \qed\end{proof}
                                  \begin{remark} If we have an unmarked Poisson
                                    process of intensity $\dif\nu(s)=\nu \dif s$,
                                    then (\ref{eq_preliminaries:1}) still
                                    holds by suppressing all occurrences of
                                    the $z$ variable.
                                  \end{remark}

                                  We are now equipped to prove the cornerstone
                                  theorem of our paper. For $u_j^{n},\,
                                  j=1,\cdots,n$ a family of elements of $
                                  L^2([0,T]\times \R^{+},\nu)$, set
                                  \begin{equation*}
                                    u^{n}(s,z,t)=\sum_{j=1}^{n}u_j^{n}(s,z)\, h_{j}^{n}(t)\text{ and } \dive{\nu} u^{n}(t)=\sum_{j=1}^{n} \dive{\nu}( u_j^{n})\ h_{j}^{n}(t).
                                  \end{equation*}
                                  For any $j\in \{1,\cdots,n\}$, let
                                  \begin{equation}\label{eq_stein:4}
                                    \xi_{j,n}^{2}=\int_{0}^{T}\int_{\R^{+}} u_{j}^{n}(s,z)^{2}\dif \nu(s,z)
                                  \end{equation}
                                  and consider
                                  \begin{equation*}
                                    \Gamma_{\xi_{n}}=\text{diag}(\xi_{j,n}^{2},\, j=1,\cdots,n).
                                  \end{equation*}
                                  Furthermore, take $Y=(Y_{j},\, j\ge 1)$ a
                                  family of independent standard Gaussian random
                                  variables 
                                  and let
                                  \begin{equation}\label{eq_fluid_queueing_BDM2:12}
                                    B_{\xi_{n}}(t)=\sum_{j=1}^{n}\xi_{j,n}\ Y_{j}\, h_{j}^{n}(t).
                                  \end{equation}
                                  \begin{theorem}
                                    \label{thm:stein_final}
                                    Assume that $(u_{j}^{n},\, j=1,\cdots,n)$ is
                                    an orthogonal family of elements of
                                    $L^{2}(\nu)$. Then, for any $f_{n}\in
                                    \Sigma_n$,
                                    \begin{multline*}
                                      \left|     \esp{f_{n}(B_{\xi_{n}})}-\esp{f_{n}(\dive{\nu} u^{n})} \right| \\
                                      \le \frac{ n^{-3/2}T^2}4
                                      \sum_{j,k,l=1}^{n}
                                      \xi_{j,n}\xi_{k,n}\xi_{l,n}
                                      \int_{[0,T]\times\R^{+}}|u_j^{n}(s,z)u_{k}^{n}(s,z)|\,
                                      |u_{l}^{n}(s,z)|\dif \nu(s,z),
                                    \end{multline*}
                                    where $\dive{\nu}$ is defined by
                                    (\ref{eq:defdive}).
                                  \end{theorem}
                                  \begin{proof}
                                    For the sake of notational simpicity, we
                                    remove the suffix $n$ as it is fixed along
                                    the proof. Note that in view of
                                    \eqref{eq_questa_BDM:4}, there is no
                                    ambiguity to denote $D_{n}$ as $D$ since
                                    they coincide on $W_{n}$. To shorten the
                                    equations, $E$ stands for $[0,T]\times
                                    \R^{+}$ and $x=(s,z)$ is a generic point of
                                    $E$.

                                    Dividing each $u_{j}^{n}$ by $\xi_{j,n}$,
                                    $j\ge 1$, it is sufficient to prove the
                                    result for $\xi_{j,n}=1$, $j\ge 1$. Now recall (\ref{eq:defBn}).
                                    First, in view of {\eqref{eq_questa_BDM:2}},
                                    \begin{multline}\label{eq_stein:5}
                                      \esp{f(B_{n})}-\esp{f(\dive{\nu} u)}
                                      =-\sum_{j=1}^{n}\int_{0}^{\infty}\esp{\dive{\nu} u_{j}\<D P_{t} f(\dive{\nu} u)\,,\,h_{j}\>_{H}}\, \dif t\\
                                      +\sum_{j=1}^{n}\int_{0}^{\infty}\esp{\<
                                        D^{(2)}P_{t}f(\dive{\nu}
                                        u)\,,\,h_{j}\otimes h_{j}\>_{H^{\otimes
                                            (2)}}}\dif t.
                                    \end{multline}
                                    According to the integration by parts
                                    formula~{\eqref{eq_preliminaries:1}} and to
                                    the fundamental theorem of calculus, we get
                                    that
                                    \begin{multline*}
                                      \sum_{j=1}^{n}\esp{\dive{\nu} u_{j}
                                        \<h_{j},\ D P_{t}f(\dive{\nu}
                                        u)\>_{\cm }}\\
                                      \shoveleft{ =\sum_{j=1}^{n} \mathbf
                                        E\Biggl[ \int_{E} u_{j}(x)\, \<D
                                        P_{t}f\bigl(\dive{\nu} u+u(x) \bigr)
                                        -D P_{t} f\bigl(\dive{\nu} u\bigr)\,,\,h_{j}\>_{\cm}\dif \nu (x)\Biggr]}\\
                                      \shoveleft{=\sum_{j,k=1}^{n}\mathbf
                                        E\Biggl[ \int_{E} \int_{0}^{1} u_{j}(x)
                                        u_{k}(x)}
                                      \<D^{(2)} P_{t}f\bigl(\dive{\nu}
                                      u^{}+r\,u(x)\,,\,h_{j}\otimes h_{k}
                                      \bigr)\>_{H^{\otimes (2)}}\!\!\!\dif r\dif
                                      \nu(x)\Biggr].
                                    \end{multline*}
                                    But as the $u_{k}$'s are orthonormal,
                                    \begin{multline*}
                                      \esp{\sum_{j=1}^{n} \<D^{(2)}
                                        \PP_{t} f\bigl(\dive{\nu} u \bigr)\,,\,h_{j}\otimes h_{j}\>_{H^{\otimes (2)}}}\\
                                      =
                                      \esp{\sum_{j,k=1}^{n}\int_{E}\int_{0}^{1}u_{j}(x)
                                        u_{k}(x)\<D^{(2)}
                                        \PP_{t}f\bigl(\dive{\nu} u
                                        \bigr)\,,\,h_{j}\otimes h_{k} \>_{\cm }
                                        \dif r\dif \nu(x)}.
                                    \end{multline*}
                                    Since $f$ belongs to $\Sigma_{n}$, the
                                    right-hand-side of \eqref{eq_stein:5}
                                    becomes
                                    \begin{multline*}
                                      \sum_{j,k=1}^{n}\int_{0}^{\infty }
                                      \int_{E} \int_{0}^{1} \esp{\< D^{(2)}
                                        \PP_{t}f\bigl(\dive{\nu} u+ru (x)\bigr)
                                        - D^{(2)} \PP_{t} f\bigl(\dive{\nu} u\bigr)\,,\,h_{j}\otimes h_{k}\>_{H^{\otimes (2)}}} \\
                                      \shoveright{\times u_{j}(x) u_{k}(x) \dif r \dif \nu(x)\dif t}\\
                                      \le \sum_{j,k=1}^{n}
                                      \|h_{j}\|_{L^{2}}\|h_{k}\|_{L^{2}}
                                      \int_{E}\|u(x)\|_{W}|u_{j}(x)\,u_{k}(x)|\dif
                                      \nu(x) \left(\int_{0}^{T} r\dif
                                        r\right)\left(\int_{0}^{\infty }
                                        e^{-2t}\dif t\right).
                                    \end{multline*}
                                    Observing that
                                    \begin{equation*}
                                      \|u(x)\|_{\fs}\le \sum_{l=1}^{n}|u_{l}(x)|\,\|h_{l}\|_{\fs}=n^{-1/2}\sum_{l=1}^{n} |u_{l}(x)|,
                                    \end{equation*}
                                    the result follows by recalling that
                                    $\|h_{j}\|_{L^{2}}\le n^{-1/2}$ for all
                                    $j\in \{1,\cdots,n\}$.
                                    \qed
                                  \end{proof}



                                  \section{Proof of Theorem \ref{thm:mainMM1}}
                                  \label{sec:proofMM1}
                                  We now turn to the proof of Theorem
                                  \ref{thm:mainMM1}. 
                                  Fix $T\le \frac{x}{\mu-\lambda}$.
                                  Then for all $n\in\N^*$ we readily have that
                                  \begin{equation}
                                    \dist_{\SP}\left(Z^{\dag}_n,B\right)
                                    \leq \dist_{\SP}\left(Z^{\dag}_n,\Pi_nZ^{\dag}_{n}\right) +\dist_{\Sigma}(\Pi_nZ^{\dag}_{n},\Pi_nB)
                                    +\dist_{\Sigma}(\Pi_n B,B).\label{eq:triangMM1}
                                  \end{equation}
                                  First observe that the function
                                  $\overline{L^\dag}$ is affine, and hence
                                  coincides with $\Pi_n \overline{L^\dag}$ on
                                  $[0,T]$. Moreover, the operator $\Pi_n$ is
                                  linear and the elements of $\SP$ are
                                  1-Lipschitz-continuous, thus we have that for
                                  all $n$,
                                  \begin{align}
                                    \dist_{\SP}(Z^{\dag}_{n},\Pi_nZ^{\dag}_{n}) \le \esp{\parallel Z^{\dag}_n - \Pi_nZ^{\dag}_n\parallel_{W}}
                                    & \le{1\over \sqrt{n(\lambda+\mu)}} \esp{\parallel L^{\dag}_n - \Pi_nL^{\dag}_n\parallel_{W}}\nonumber\\
                                    & \le {c \log n \over \log \log n\,\sqrt{n}},\label{eq:MM1_1}
                                  \end{align}
                                    where the last inequality follows from applying Lemma \protect{\ref{thm:BD}}
    to the Markov processes $\suite{L^{\dag}}{n}{1}$ for $J\equiv 1$ and $\alpha
    \equiv \lambda\vee
    \mu$. 
    Now, for any $n\in \mathbb N^*$, if we let $\tau^n_0=\inf\{t>0,\,
    L^{\dag}_{n}(t)=0\}$, for any $F \in \Sigma$ we have that
    \begin{multline}
      \E\left[\left|F\left(\Pi_nZ^{\dag}_n\right) - F\left(\Pi_nB\right)\right|\right]\\
      \shoveleft{=\E\left[\left|F\left(\Pi_nZ^{\dag}_n\right) -
          F\left(\Pi_nB\right)\right|\mathbf 1_{\{T < \tau^n_0\}}\right]}\\
      +\E\left[\left|F\left(\Pi_nZ^{\dag}_n\right) -
          F\left(\Pi_nB\right)\right|\car_{\{T\ge \tau^n_0\}}\right].
      \label{Conditioning of Z^n}
    \end{multline}
    We first prove that for some $c>0$,
    \begin{equation}
      \label{eq:MM1before}
      \E\left[\left|F\left(\Pi_nZ^{\dag}_n\right) - F\left(\Pi_nB\right)\right|\mathbf 1_{\{T<\tau^n_0\}}\right]\leq {c \over \sqrt{n}},\,\,n\in\N^*.
    \end{equation}
    Fix $n\in \N^*$. On the event $\{T < \tau_0^n\}$, for any
    $t\in \left[0,T\right)$ we have that
    \begin{align*}
      Z^{\dag}_{n}(t)
         &=\frac{1}{\sqrt{\lambda+\mu}}\left(\sqrt{\lambda}\left(\frac{\Pois_{n\lambda}(t)}{\sqrt{\lambda n}}-\sqrt{\lambda n} t \right)
           -\sqrt{\mu}\left(\frac{\Pois_{n\mu}(t)}{\sqrt{\mu n}}-\sqrt{\mu n} t\right)\right)\\
         &=: \frac{1}{\sqrt{\lambda+\mu}}\left(Z^{\dag}_{\lambda, n}(t) - Z^{\dag}_{\mu, n}(t)\right).
    \end{align*}
    To apply Theorem \ref{thm:stein_final}, it is useful to represent the processes $Z^\dag_n$, $n\ge 1$ as marked Poisson processes.
    For this, we fix $n\in\N^*$, and let $N_{n(\lambda+\mu)}^{\dag}$ be the marked Poisson point process on
    $[0,T] \times \{-1,1\}$ of control measure
    \begin{equation*}
      \dif\nu^{\dag}_{n}(s,r)=n(\lambda+\mu) \dif s\otimes \Bigl(\frac{\lambda}{\lambda+\mu}\,\varepsilon_{1}(\!\dif r)+\frac{\mu}{\lambda+\mu}\,\varepsilon_{-1}(\!\dif r)\Bigr),
    \end{equation*}
    that is, an ordinary Poisson process on the positive half-line with intensity $n(\lambda+\mu)$,
    such that each atom is assigned a mark $+1$ or $-1$, independently of everything else,
    with respective probability $\lambda(\lambda+\mu)^{-1}$ and $\mu(\lambda+\mu)^{-1}$. By the thinning
    property of Poisson processes, the point process counting the atoms of
    $N_{n(\lambda+\mu)}^{\dag}$ with mark $+1$ (respectively $-1$) is Poisson
    of intensity $n\lambda$ (respectively $n\mu$).
    For any $t\in[0,T]$, let
   \[v_{t}\, :\,
      \left\{\begin{array}{ll}
      [0,T]\times \{-1,1\}&\longrightarrow \R\\
      (s,r)&\longmapsto \frac{1}{\sqrt{n(\lambda+\mu})}\ r\, \car_{[0,t)}(s),
    \end{array}\right.\]
    and define for all $i=1,\cdots,n$,
    \begin{equation*}
      u^{\dag}_{i}(s,r)=\sqrt{\frac{n}{T}}\Bigl(v_{t^n_i}(s,r)- v_{t^n_{i-1}}(s,r)\Bigr)
      = \frac{1}{\sqrt{T(\lambda+\mu)}}\ r\, \car_{\left[t^n_{i-1},t^n_i\right)}(s).
    \end{equation*}
    Then, it is easily checked that
    \begin{equation*}
      Z^{\dag}_{n}(t) \stackrel{\text{dist}}{=}\Delta^*_{\nu^{\dag}_{n}}v_{t},\,t\le T,
    \end{equation*}
    which, recalling (\ref{eq:interpol}) and (\ref{eq:defhni}), yields to
     \begin{equation*}
      \Pi_nZ^{\dag}_{\lambda, n} \stackrel{\text{dist}}{=}\sum_{i=1}^{n}\Delta^*_{\nu^{\dag}_{n}}u^{\dag}_{i} \ h_{i}^{n}.
    \end{equation*}
    It is then clear that for all $i,j \le n$,
    \begin{equation*}
      \int_{[0,T]\times \{-1,1\}}u^{\dag}_{i}(s,r)u^{\dag}_{j}(s,r)\dif \nu^{\dag}_{n}(s,r)=\delta_{ij},
    \end{equation*}
    so $\{u^\dag_i,\,i=1,\cdots,n\}$ is an orthogonal family.
    Moreover, comparing (\ref{eq:interpol}) to (\ref{eq_fluid_queueing_BDM2:12}), we readily obtain that $\Pi_nB \stackrel{\text{dist}}{=} B_{\xi^\dag}$ when letting $\xi^{\dag}_{j,n}=1$ for all
    $j=1,\cdots,n$.
    Consequently, (\ref{eq:MM1before}) follows from Theorem \protect{\ref{thm:stein_final}} and the fact that
    \begin{equation*}
      \sum_{j,k,l=1}^{n} \int_{E}
      |u^{\dag}_ju^{\dag}_{k}u^{\dag}_{l}| \dif \nu^{\dag}_{n}=\frac{1}{T^{3/2}\sqrt{\lambda+\mu}}\sum_{i=1}^{n} \int_{t^n_{i-1}}^{t^n_i} n \dif s=\frac{n}{\sqrt{T(\lambda+\mu)}}\cdotp
    \end{equation*}
    Regarding the second term on the right-hand side of \eqref{Conditioning of
      Z^n}, observe that $F$ is in particular bounded, so there
    exists a constant $c'$ such that for all $n\in\N^*$,
    \[\E\left[\left|F(\Pi_nZ^{\dag}_{n})-F(\Pi_nB)\right| \car_{\{T >
          \tau^n_0\}}\right] \le c\, \pr{T> \tau^n_0}.\] But $\pr{T > \tau_0^n}$
    tends to 0 with exponential speed from Theorem 11.9 of
    \cite{shwartz_large_1995}: if $\rho<1$, for any $x>0$ and any $y<0$,
$$\lim_{n \rightarrow \infty} \frac{1}{n}\ \mbox{log}\ \pr{\tau_0^n \leq \frac{x}{\lambda-\mu}+y}=-f(y),$$
where $f$ is strictly positive on $(0,\infty)$. This shows that for some $c''$,
$$\E\left[\left|F(Z^{\dag}_{n})-F(B)\right| \car_{\{T > \tau^n_0\}}\right]\leq c''e^{-n}$$
for all $n$ which, together with (\ref{eq:MM1before}) in \eqref{Conditioning of
  Z^n}, shows that for some constant $c$, for all $n\in\N^*$,
\begin{equation*}
  \dist_{\Sigma}(\Pi_nZ^{\dag}_{n},\Pi_nB) \le {c \over \sqrt{n}}.
\end{equation*}
This, together with (\ref{eq:MM1_1}) and (\ref{dist B, Bgamma}) in
(\ref{eq:triangMM1}), concludes the proof.

\section{Proof of Theorem \ref{thm:mainMMinfty}}
\label{sec:proofMMinfty}

We now turn to the speed of convergence in the diffusion approximation of the infinite server queue. Fix $T>0$ throughout this section.

\subsection{An integral transformation}
We know from eq. (6.23) of \cite{robert_stochastic_2003} that
the sequence of processes $\suite{Y^{\sharp}}{n}{1}$ defined for all $n\ge 1$ by
  \begin{equation}
    \label{eq:defYN}
    t\mapsto Y^{\sharp}_n(t) := Z^\sharp_n(t)-Z^{\sharp}_n(0) +\mu \int_{0}^{t}Z^{\sharp}_n(s)\dif s
  \end{equation}
  converges in distribution to the time-changed standard brownian motion $B\circ \gamma$, where
  \begin{equation}
    \label{eq:defgamma}
    \gamma (t)=2\lambda t -{\lambda \over \mu}(1-e^{-\mu t}),\quad t\ge 0.
  \end{equation}
  This integral transformation of the processes $Z^\sharp_n$, $n\ge 1$ will turn out to be useful
  to bound the rate of convergence of $\{Z^\sharp_n\}$ to the Ornstein-Uhlenbeck process $Z^\sharp$ defined by (\ref{ProcessusX}).
  Specifically, as will be shown below, the latter rate of convergence
  is in fact bounded by that of $\{Y^\sharp_n\}$ to the time-changed brownian motion $B\circ\gamma$.
   First observe that
\begin{proposition}
  \label{prop:bijTheta}
  The mapping
  \[\Theta:\left\{\begin{array}{ll}
                    \D&\longrightarrow \R \times \D^0_T\\
                    f &\longmapsto \left(f(0)\,,\,f(.)-f(0)+\mu \displaystyle\int_{0}^{.}f(s)ds\right)
                  \end{array}
                \right.\] is linear, continuous (for the Skorohod topology on $\D$), and one to one.
            \end{proposition}

            \begin{proof} Let us fix $\eta\in \D^0_T$ and consider the
              following integral equation of unknown function $z$,
              \begin{equation*}
                z(t)-z(0)=- \mu\int_{0}^{t}z(s)\dif s+\eta(t) .
              \end{equation*}
              We clearly have for all $t\ge0$,
$$
z(t)=z(0)e^{-\mu t}+\eta(t)-\mu\int_{0}^{t}e^{-\tau(t-s)}\eta(s)\dif s,
$$
hence $\Theta$ is bijective and for all $(x,\eta)\in \R \times \D^0_T$,
\begin{eqnarray}
  \Theta^{-1}(x,\ \eta)=\left(t \longmapsto xe^{-\mu t}+\eta(t)-\mu\int_{0}^{t}e^{-\mu(t-s)}\eta(s)\dif s\right).
\end{eqnarray}
Linearity and continuity are then straightforward.
\hfill \qed \end{proof}
 Also,
\begin{lemma}
  On the subset of $\{0\}\times\Theta(\D)$ whose image
  by $\Theta^{-1}$ is in $\D$, $\Theta^{-1}$ is linear
  and continuous.
\end{lemma}
\begin{proof}
  For all $\eta$, $\omega\in \Theta(\D)$ and all $t\le T$, we have that
  $$ \Theta^{-1}(0,\eta)(t)-\Theta^{-1}(0,\omega)(t)=\eta(t)-\omega(t)-\mu\int_{0}^{t}e^{-\mu(t-s)}(\eta(s)-\omega(s))\dif s.$$
Hence, by an immediate change of variable we get that
$$\parallel \Theta^{-1}(0,\eta)-\Theta^{-1}(0,\omega)\parallel_{W}<\parallel\eta-\omega\parallel_{W}+\mu\parallel\eta-\omega\parallel_{W}\sqrt{\int_{0}^{T}e^{-2\mu s}\dif s},$$
so that for some positive constant $k$,
$$ \parallel
\Theta^{-1}(x,\eta)-\Theta^{-1}(y,\omega)\parallel_{W}<k\parallel\eta-\omega\parallel_{W}.$$
This completes the proof.
\hfill \qed\end{proof}
We obtain the following,
\begin{corollary}
  These exists a  positive constant $c$ such that  that for all $n\in\N^*$,
  \[\dist_{\Sigma}(\Pi_nZ^{\sharp}_n,Z^{\sharp})\leq c
    \,\dist_{\Sigma}(\Pi_nY^{\sharp}_n,B\circ\gamma).\]
  \label{Corollary Theta}
\end{corollary}

\begin{proof}
 In view of the weak convergence $Z^{\sharp}_{n} \Rightarrow B\circ\gamma$,
 the linearity and continuity of $\Theta$ and the Continuous Mapping Theorem,
 we have the weak convergence
  \[\Theta(Z^{\sharp}_{n})=(0, Y^{\sharp}_{n}) \Rightarrow
    \left(0,B\circ\gamma\right).\]
    However, expression (6.34) in \cite{robert_stochastic_2003} shows that
  for all $t$, $\Theta\left(Z^{\sharp}\right)=\left(0,B\circ\gamma\right)$
  which, together with the linearity of $\Theta$ and of the operator $\Pi_n$ for
  all $n$, concludes the proof.
\hfill \end{proof}

\subsection{Alternative representation}
With Corollary \ref{Corollary Theta} in hand, we are rendered to assess the rate of convergence
of $\suite{Y^\sharp}{n}{0}$ to the time-changed brownian motion $B\circ \gamma$.
For that purpose, we aim at applying again Theorem \ref{thm:stein_final} and, as above,
it is useful for this to view the processes $L^{\sharp}_n$, $n\ge 1$
as simple functions of marked Poisson processes.

Specifically, following Section 7.2 of \cite{robert_stochastic_2003}, we have the
following alternative representation of the process $L^\sharp$:
A point $(x,z)$ represents a customer arriving at time $x$ and requiring a
service of duration $z$, and we let $\Pois_{\lambda, \mu}$ be a Poisson process on
$\mathbb{R}^{+}\times \mathbb{R}^{+}$ of control measure $\lambda \dif x\otimes
\mu e^{-\mu z}\dif z$. At any time $t\ge 0$, the number of busy servers at $t$
equals the number of points located in the shaded trapeze bounded by the axes of
equation $x=0$ and $x=t$, and above the line $z=t-x$: in other words,
$$L^{\sharp}(t)=\int_{C_{t}}\dif \Pois_{\lambda, \mu}(x,z),\,t\ge 0,$$ 
where
\begin{equation}
  \label{eq:defCt}
  C_{t}=\{(x,z), 0 \leq x \leq t, z\geq t-x\}.
\end{equation}


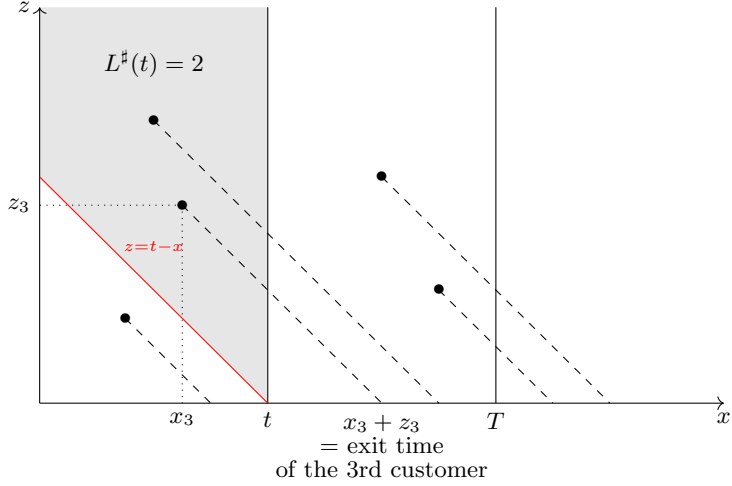
\begin{figure}[!ht]
  \label{fig=representation}
  \begin{center}
    \begin{tikzpicture}[font=\fontsize{9}{9}\selectfont,scale=0.75]

      \fill[color=black!10] (0,7) -- (0,4) -- (4,0) -- (4,7) --cycle; \draw[->]
      (0,0)--(12,0) node[below] {$x$} ; \draw[->] (0,0)--(0,7) node[left] {$z$};
      \draw[dashed] (7,2) -- (9,0);

      \draw (4,0) node[below] {$t$} -- (4,7); \draw[dotted](2.5,3.5) --
      (2.5,0)node[below] {$x_{3}$}; \draw[dotted](2.5,3.5) -- (0,3.5)node[left]
      {$z_{3}$};
    \draw (6,0.2) -- (6,-0.5) node[below] {
        \begin{tabular}{c}
          $x_{3}+z_{3}$\\
          = exit time
          of the 3rd customer
        \end{tabular} };
      \draw[color=red] (0,4) -- node[above]
      {\raisebox{0.4cm}{$\scriptstyle{z=t-x}$}} (4,0) ; \draw[dashed]
      (2.5,3.5)-- (6,-0) ;
      \draw[dashed] (2,5)-- (7,0) ;
      \draw[dashed] (6,4)-- (10,0) ;
      \draw[dashed] (1.5,1.5)-- (3,0) ;
      \draw (8,7) -- (8,0) node[below] {$T$};
      \foreach \Point in {(2.5,3.5), (2,5), (6,4), (1.5,1.5), (7,2)}{
        \node at \Point {\textbullet};}
      \node at (2,6)  {$L^{\sharp}(t)=2$};
    \end{tikzpicture}
    \caption{Representation of the $M/M/\infty$ queue}
  \end{center}
\end{figure}

Fix a positive integer $n$ throughout this section. After scaling, for all $t\ge
0$ we get that
$$\overline {L^{\sharp}_n}(t)=\frac{1}{n}\Pois_{\lambda n,\mu}(C_{t}).$$
\noindent Let us denote for all $(x,z)$ in the positive orthant by
\[\dif\nu^{\sharp}_n(x,z):=\lambda n \dif x\otimes \mu e^{-\mu z}\dif z,\]
the control measure of $\Pois_{n\lambda,\mu}$. As readily follows from
(\protect{\ref{eq:limfluinfty}}), the fluid limit $\overline{L^{\sharp}}$ can be
written as
\[\overline{L^{\sharp}}(t)= {1\over n}\int \mathbf{1}_{C_{t}}(x,z)\dnn(x,z),\,t\ge 0,\]
in a way that
\begin{equation}
  Z^{\sharp}_{n}(t)=\frac{1}{\sqrt{n}}\int\mathbf{1}_{C_{t}}\left( \dif \Pois_{\lambda
      n,\mu}-\dnn \right),\,t\ge 0,
\end{equation}
for $C_{t}$ defined by (\protect{\ref{eq:defCt}}). We deduce that for all $t\ge
0$,
\begin{multline}
  Y^{\sharp}_{n}(t) = \,\frac{1}{\sqrt{n}}\int\mathbf{1}_{C_{t}}\left(\dif\Pois_{\lambda n,\mu}-\dnn\right)
 +\mu
    \int_{0}^{t}\frac{1}{\sqrt{n}}\int\mathbf{1}_{C_{s}}\left(\dif\Pois_{\lambda
        n,\mu}- \dnn\right)\dif u\\
    =\,\frac{1}{\sqrt{n}}\, \dive{\lambda n,\mu}( \mathbf{1}_{C_{t}})+ \mu
    \int_{0}^{t}\frac{1}{\sqrt{n}}\dive{\lambda n,\mu}( \mathbf{1}_{C_{u}})\dif
    u, \label{Y_NGeometrique}
\end{multline}
where $\dive{\lambda n,\mu}$ is defined by (\ref{eq:defdive}).


\subsection{Reduction to the finite dimension}
\label{subsec:MMinfty}

Fix $n\in \N^*$ and recall (\ref{eq:interpol}).
It follows from \eqref{Y_NGeometrique} that
\begin{align*}
  \Pi_nY^{\sharp}_{n}
                     &= \sum_{i=1}^{n}{1\over \sqrt{T}}\left(\Dive{\lambda n,\mu}{\mathbf{1}_{C_{t^n_i}}-\mathbf{1}_{C_{t^n_{i-1}}}}+ \mu \int_{t^n_{i-1}}^{t^n_i}\dive{\lambda n,\mu}(\mathbf{1}_{C_{u}})\dif u\right)h^n_i\\
                     &=\sum_{i=1}^{n}\dive{\lambda n,\mu} (u^{\sharp}_i)\,h^n_i,
\end{align*}
where for all $i=1,\cdots,n$ and all $(x,z)\in{\mathbb{R}^2}$,
\begin{equation}
  u^{\sharp}_i(x,z) =\frac{1}{\sqrt{T}}\left( \mathbf{1}_{C_{t^n_i}}(x,z)-\mathbf{1}_{C_{t^n_{i-1}}}(x,z)+\mu\int_{t^n_{i-1}}^{t^n_i}\mathbf{1}_{C_{u}}(x,z)\du \right).
  \label{eq:defuMMinfty}
\end{equation}
Let us denote for any
$i=1,\cdots,n$, 
$$\xi^{\sharp}_{i,n}:=\sqrt{\gamma\left(t^n_i\right)-\gamma\left(t^n_{i-1}\right)}.$$
The following result is proven in appendix \protect{\ref{subsec:proofMMinfty}},
\begin{proposition}
  \label{prop:ui}
  For any $n$, the family $\left(u^{\sharp}_i,\,i=1,\cdots,n\right)$ has the
  following properties:
  \begin{enumerate}
  \item[(i)] It is orthogonal in $L^{2}\left(\nu^{\sharp}_n\right)$;
  \item[(ii)] For some constant $c$ independent of $n$,
    \begin{equation*}
      \sum_{i=1}^{n} \sum_{j=1}^{n}\sum_{k=1}^{n} \int_{E}
      |u^{\sharp}_iu^{\sharp}_{j}u^{\sharp}_{k}|\dnn\leq n c.
    \end{equation*}
  \item[(iii)] For any $i \in \{1,\cdots,n\}$,
    \begin{equation*}
      \int\int u^{\sharp}_iu^{\sharp}_i\dnn=\frac{n}{T}\,(\xi^{\sharp}_{i,n})^2.
    \end{equation*}
  \end{enumerate}
\end{proposition}
\noindent Notice that for a large enough $n$, for all $t\ge 0$,
\begin{equation*}
  \frac{n}{t}\,(\xi^{\sharp}_{i,n})^2\xrightarrow{i/n\to t} \gamma'(t) \text{ and for a fixed $i$, } \frac{n}{t}\,(\xi^{\sharp}_{i,n})^2\xrightarrow{n\to \infty} \gamma'(0).
\end{equation*}

\noindent We thus have the following result,
\begin{proposition}
  For some $c$, for all positive integer $n$, the respective interpolations of
  $Y^{\sharp}_n$ and $B\circ\gamma$ satisfy
$$\dist_{\Sigma}(\Pi_n Y^{\sharp}_{n}, \Pi_n (B\circ\gamma))\leq \frac{c}{\sqrt{n}}.$$
\label{Rate of convergence YNtilde towards BgammaN}
\end{proposition}
\begin{proof}Fix $n\in\N^*$.
  It is an immediate consequence of (\ref{eq:interpol}) and (\ref{eq_fluid_queueing_BDM2:12}) that
$$\pi_{n}(B\circ{\gamma})\stackrel{\text{dist}}{=}\sum_{j=1}^{n}Y^{\sharp}_{j}\, h_{j}^{n}=B_{\xi^{\sharp}},$$
where 
$\left(Y^{\sharp}_{k},\, k=1,\cdots,n\right)$ is a family of independent
centered Gaussian random variables such that
$\text{var}(Y^{\sharp}_{k})=(\xi_{k}^{\sharp})^{2}$ for all $k$. From assertion
(i) of Proposition \protect{\ref{prop:ui}}, we can apply Theorem
\protect{\ref{thm:stein_final}}~: for any $f\in \Sigma$,
\begin{equation*}
  \left|     \esp{f(B_{\xi^{\sharp}})}-\esp{f(\dive{\lambda n,\mu} (u^{\sharp}))} \right|
  \leq  \ {n^{-3/2} T^2\over 4} \sum_{j,k,l=1}^{n} \int_{E}
  |u^{\sharp}_ju^{\sharp}_{k}|\, |u^{\sharp}_{l}|\dnn.
\end{equation*}
Assertions (ii) and (iii) of Proposition \protect{\ref{prop:ui}} allow us to
conclude.
\hfill \end{proof}

\subsection{Proof of Theorem \ref{thm:mainMMinfty}}
\label{subsec:speed}
We are now in position to prove Theorem \ref{thm:mainMMinfty}.  For all $n\in\N^*$,.
We have that
  \begin{multline}
    \label{eq:triangMMinfty}
    \dist_{\SP}(Z^{\sharp}_n,Z^{\sharp})\\
    \begin{aligned}
      &\le \dist_{\SP}(Z^{\sharp}_n,\Pi_nZ^{\sharp}_{n})+\dist_{\Sigma}(\Pi_nZ^{\sharp}_{n},Z^{\sharp})\\
      &\le \dist_{\SP}(Z^{\sharp}_n,\Pi_nZ^{\sharp}_{n})+c\dist_{\Sigma}(\Pi_nY^{\sharp}_{n},B\circ\gamma)\\
      &\le
      \dist_{\SP}(Z^{\sharp}_n,\Pi_nZ^{\sharp}_{n})+c\dist_{\Sigma}(\Pi_nY^{\sharp}_{n},\Pi_n
      B\circ{\gamma})+c\dist_{\Sigma}(\Pi_nB\circ{\gamma},B\circ\gamma),
    \end{aligned}
  \end{multline}
  where we applied Corollary \protect{\ref{Corollary Theta}} in the second
  inequality. Now define the stopping times
  \[\tau^{\sharp}_n = \inf\left\{t \ge 0\,:\,N_{n\lambda}(t) \ge 2\lambda n
      T\right\},\,n\in\N^*.\] Then, as all functions of $\SP$ are bounded and
  Lipschitz continuous we obtain that for all $n$,
  \begin{multline}
    \dist_{\SP}\left(Z^{\sharp}_n,\Pi_nZ^{\sharp}_{n}\right)\\
    \begin{aligned}
      &\le \sup_{F \in \SP}
      \esp{\left|F\left(Z^{\sharp}_n\right)-F\left(\Pi_nZ^{\sharp}_n\right)\right|\mathbf
        1_{\left\{T < \tau^{\sharp}_n\right\}}}
      + c\, \pr{T \ge \tau^{\sharp}_n}\\
      &\le \esp{\parallel Z^{\sharp}_n - \Pi_nZ^{\sharp}_n\parallel_{W} \mathbf 1_{\left\{T < \tau^{\sharp}_n\right\}}}+ c \,\pr{T \ge \tau^{\sharp}_n}\\
      &\le \esp{\parallel Z^{\sharp}_n\left(.\wedge
          \tau^{\sharp}_n\right)-\Pi_n\left(Z^{\sharp}_n\left(.\wedge
            \tau^{\sharp}_n\right)\right)\parallel_{W}\mathbf 1_{\left\{T
            < \tau^{\sharp}_n\right\}}}+ c\, \pr{T \ge \tau^{\sharp}_n}.
    \end{aligned}
    \label{eq:MMinfty_0}
  \end{multline}
  On the one hand, from Tchebychev inequality we have that for all $n$,
    \begin{equation}
    \label{eq:MMinfty_1}
    \pr{T \ge \tau^{\sharp}_n} = \pr{N_{n\lambda}(T) \ge 2\lambda n T} \le {\mbox{Var}\,(N_{n\lambda}(T)) \over (\lambda n T)^2} \le { c \over n}\cdotp
  \end{equation}
  Also, for any $n$, on $\{T < \tau^{\sharp}_n\}$ we have that
  \[L^{\sharp}_n\left(t\wedge \tau^{\sharp}_n\right) \le N_{n\lambda}(t) \le
    2\lambda nT,\] therefore the Markov process $L^\sharp_n\left(.\wedge
    \tau^\sharp_n\right)$ satisfies to the Assumptions of Lemma \ref{thm:BD} for
  $J \equiv 1$ and $\alpha \equiv \lambda \vee (\mu T)$. Thus we obtain as in
  (\ref{eq:MM1_1}) that for all $n$,
  \begin{multline}
    \label{eq:MMinfty_2}
    \esp{\parallel Z^{\sharp}_n\left(.\wedge \tau^{\sharp}_n\right)-\Pi_n\left(Z^{\sharp}_n\left(.\wedge \tau^{\sharp}_n\right)\right)\parallel_{W}\mathbf 1_{\left\{T < \tau^{\sharp}_n\right\}}}\\
    \begin{aligned}
      &\le {1\over \sqrt{n}} \esp{\parallel L^{\sharp}_n - \Pi_nL^{\sharp}_n\parallel_{W}} + \sqrt{n}\parallel \overline{L^{\sharp}} - \Pi_n\overline{L^{\sharp}}\parallel_{W}\\
      & \le {c \log n \over \log \log n\ \sqrt{n}},
      \cdotp \end{aligned}\end{multline}
      where, recalling (\ref{eq:limfluinfty}), we use the fact that
  \begin{align*}
    \sqrt{n}\parallel \overline{L^{\sharp}} - \Pi_n\overline{L^{\sharp}}\parallel_{W}
    &\le 2\sqrt{n} \max_{i\in [0,n-1]}\sup_{t\in \left[t^n_{i};\frac{(i+1){T}}{n}\right]}\left|e^{-\mu t} - e^{-\mu {iT \over n}}\right|\\
    &\le 2\sqrt{n} \left(e^{-{\mu \over n}} - 1\right) \le {c \over \sqrt{n}}.
  \end{align*}
  Finally, gathering (\ref{eq:MMinfty_2}) with (\ref{eq:MMinfty_1}) in
  (\ref{eq:MMinfty_0}) entails that for all $n$,
  \[\dist_{\SP}(Z^{\sharp}_n,\Pi_nZ^{\sharp}_{n}) \le {c \log\,n \over
      \sqrt{n}}\] which, together with with Proposition \protect{\ref{Rate of
      convergence YNtilde towards BgammaN}} and (\protect{\ref{dist B, Bgamma}})
  in (\ref{eq:triangMMinfty}), concludes the proof.
\hfill

\appendix
\section{Moment bound for Poisson variables}
\label{sec:Poissonbound}
By following closely Chapter 2 in \cite{BLM}, we show hereafter a moment bound
for the maximum of $n$ Poisson variables. (Notice that, contrary to Exercise
2.18 in \cite{BLM} we do not assume here that the Poisson variables are
independent.)
\begin{proposition}
  \label{prop:momentPoisson}
  Let $n\in\N$ and let $X_i,\, i=1,\cdots,n$ be Poisson random variables of
  parameter~$\nu$. Then for some $c$ depending only on $\nu$ we have that
  \begin{equation}
    \label{eq:defphi}
    \esp{\max_{i=1,\cdots,n}X_i} \le c\, {\log n \over \log\log n}\cdotp
  \end{equation}
\end{proposition}
\begin{proof}
  Denote for all $i$, $Z_i=X_i-\nu$, and by $\Psi_{Z_i}$ the moment generating
  function of $Z_i$. By Jensen's inequality and the monotonicity of $\exp(.)$ we
  get that
  \[
    \exp\left(u\esp{\max_{i=1,\cdots,n}Z_i}\right)\leq\esp{\max_{i=1,\cdots,n}\exp(uZ_i)}\leq\sum_{i=1}^n\esp{\exp(uZ_1)}\leq
    n\exp\left(\Psi_{Z_i}(u)\right).\] After a quick algebra, this readily
  implies that
  \begin{align*}
    \esp{\max_{i=1,\cdots,n}Z_i}\leq \inf_{u\in\R}\left(\frac{\log n+\nu\left(e^u-u-1\right)}{u}\right)
    = \frac{\log n+\nu\left(e \frac{a}{W(a)}-1-W(a)-1\right)}{1+W(a)},
  \end{align*}
  where $W$ is the so-called Lambert function, solving the equation
  $W(x)e^{W(x)}=x$ over $[-1/e,\infty]$, and $a=\frac{\log(n/e^{\nu})}{e^\nu}$.
  This entails in turn that
$$\esp{\max_{i=1,\cdots,n}X_i}\leq \nu e \frac{a}{W(a)}-\nu+\nu=\frac{\log{(n/e^{\nu})}}{W(\log(n/e^{\nu})/e^{\nu})}\cdotp$$
We conclude by observing that $W(z)\geq \log (z) - \log\log(z)$ for all $z >e$.
Therefore there exists $c>0$ such that for $n\geq \exp
\left(e^{\nu+1}+\nu\right)$,
\[\esp{\max_{i=1,\cdots,n}X_i}
  \leq\frac{\log{(n/e^{\nu})}}{\log(\log(n/e^{\nu})/e^{\nu})-\log\log(\log(n/e^{\nu})/e^{\nu})}
  \le c\, {\log n \over \log\log n},\] which completes the proof.
 \end{proof}

\section{Proof of Proposition \protect{\ref{prop:ui}}}
\label{subsec:proofMMinfty}
Fix $n$ throughout this section, and
denote for all $i=0,...,n-1$ and $(x,z) \in\R^2$,
\[\alpha_{i}(x,z)=\mathbf 1_{C_{{t^n_i}}}(x,z),\quad\quad
  \beta_i(x,z)=\int_{t^n_i}^{t^n_{i+1}}\mathbf{1}_{C_{u}}(x,z)\du.\]

\begin{proof}[Proof of (i)]
  Recall (\protect{\ref{eq:defuMMinfty}}), and fix two indexes $0\le i <j \le
  n-1$. We have that
  \begin{multline}
    \label{eq:ortho0}
    \int\int u^{\sharp}_iu^{\sharp}_j\dnn
    = \int\int \left(\alpha_{i+1}- \alpha_i\right)\left(\alpha_{j+1}-\alpha_j \right)\dnn\\
    +\mu \int\int\beta_i\left(\alpha_{j+1}-\alpha_j\right) \dnn +\mu
    \int\int\beta_j\left(\alpha_{i+1}-\alpha_i\right)\dnn
    \shoveright{+\mu^2 \int\int\beta_j\beta_j\dnn}\\
    =: I_1 +I_2 +I_3+I_4,
  \end{multline}
  where straightforward computations show that
  \begin{align*}
    I_1
    &=\lambda n\left(2 e^{-\mu(t^n_j-t^n_i)} - e^{-\mu(t^n_j-t^n_{i+1})} - e^{-\mu(t^n_j-t^n_{i-1})}\right);\\
    I_2 
    &={\lambda n\over \mu}\left(2 e^{-\mu(t^n_j-t^n_i)} - e^{-\mu(t^n_j-t^n_{i+1})} - e^{-\mu(t^n_j-t^n_{i-1})}\right) - \lambda \left(e^{-\mu t_{j+1}^n} - e^{-\mu t^n_j}\right);\\
    I_3 
    &={\lambda n\over \mu}\left(-2 e^{-\mu(t^n_j-t^n_i)} + e^{-\mu(t^n_j-t^n_{i+1})} + e^{-\mu(t^n_j-t^n_{i-1})}\right);\\
    I_4 
    &={\lambda n\over \mu}\left(-2 e^{-\mu(t^n_j-t^n_i)} + e^{-\mu(t^n_j-t^n_{i+1})} + e^{-\mu(t^n_j-t^n_{i-1})}\right)+ \lambda \left(e^{-\mu t_{j+1}^n} - e^{-\mu t^n_j}\right).
  \end{align*}
  Adding up the above in (\protect{\ref{eq:ortho0}}) yields the result.
\hfill\qed \end{proof}

\begin{proof}[Proof of (ii)]
  For all $0 \le i,j,k \le n-1$ we write
  \begin{multline}
    \label{eq:I}
    I_{i,j,k} :=\int_{\R^2}
    |u^{\sharp}_iu^{\sharp}_{j}u^{\sharp}_{k}|\dnn \leq \int \left|(\alpha_{i+1} - \alpha_{i})(\alpha_{j+1} - \alpha_{j})(\alpha_{k+1}-\alpha_{k})\right| \dnn\\
    + \int \left|(\alpha_{i+1} - \alpha_i)(\alpha_{j+1} -
      \alpha_j)\mu\beta_k\right|\dnn + \int \left|(\alpha_{j+1} -
      \alpha_j)(\alpha_{k+1} - \alpha_k)\mu\beta_i\right|\dnn\\ + \int
    \left|(\alpha_{i+1} - \alpha_i)(\alpha_{k+1} -
      \alpha_k)\mu\beta_j\right|\dnn
    +\int \left|\left(\alpha_{i+1} - \alpha_i\right)\mu^2\beta_j\beta_k\right| \dnn\\
    +\int \left|\left(\alpha_{j+1} - \alpha_j\right)\mu^2\beta_i\beta_k\right|
    \dnn
    +\int \left|\left(\alpha_{k+1} - \alpha_k\right)\mu^2\beta_i\beta_j\right| \dnn\\
    +\int \left|\mu^3\beta_i\beta_j\beta_k\right| \dnn =: \sum_{l=1}^8
    I_{i,j,k}^l.
  \end{multline}
  It can be easily retrieved that
  \begin{align*}
    I_{i,i,i}^1 &=n\left(\frac{\lambda}{n}-\frac{\lambda}{\mu}\left(1-e^{-\frac{\mu T}{n}}\right)\left(1-e^{\mu T{i+1 \over n}}\right)\right)\leq {\lambda \over \mu };\\
    I_{i,j,k}^1 &=0,\quad 1\le i<j<k \le n;\\
    I_{i,i,k}^1 &={\lambda n\over \mu }\left(e^{\mu t^n_{i+1}} - e^{\mu t^n_{i}}\right)\left(e^{-\mu t^n_{k}} - e^{-\mu t^n_{k+1}} \right) \leq {\lambda T^2\over \mu n},\quad i=j<k,
  \end{align*}
  and the other cases can be treated similarly. Also, simple computations show
  that if $i<j$,
  \begin{equation*}
    \mu\int \left|(\alpha_{i+1} - \alpha_i)(\alpha_{j+1} - \alpha_j)\beta_k\right| \dnn
    \leq\lambda\left(e^{\mu t^n_{i+1}} - e^{\mu t^n_i}\right)\left(e^{-\mu t^n_j} - e^{-\mu t^n_{j+1}} \right) \leq {\lambda T^2\over n^2},
  \end{equation*}
  whereas if $i=j$, the above integral is upper bounded by
  \[2\lambda T\left(2+e^{-\mu t^n_{i+1}}-e^{-\mu t^n_i}-2e^{-\frac{\mu T}{n}}\right)\leq\frac{2\lambda\mu T^2}{n}T^2.\]
  It readily follows that in all cases, $I_{i,j,k}^2,I_{i,j,k}^3$ and
  $I_{i,j,k}^4$ are less than ${c\, n^{-1}}$ for some constant $c$.
  \noindent Reasoning similarly, we also obtain that for all $i,j,k$,
  \begin{equation*}
    \mu^2\int \left|(\alpha_{i+1} - \alpha_i)\mu^2\beta_j\beta_k\right| \dnn\leq\frac{\mu^2}{n}\left(\frac{\lambda }{n}-\frac{\lambda}{\mu}\left(1-e^{-\frac{\mu T}{n}}\right)\left(1-e^{\mu T{i+1 \over n}}\right)\right)\leq {\lambda \over \mu n^2}T,
  \end{equation*}
  so that in all cases the $I_{i,j,k}^5,I_{i,j,k}^6$ and $I_{i,j,k}^7$'s are
  less than ${c \, n^{-2}}$ for some $c$. Finally, observing that for all
  $u,v,w$,
$$\int\int\mathbf{1}_{C_u}\mathbf{1}_{C_v}\mathbf{1}_{C_w}\lambda\mu e^{-\mu y}\dx\dy=\frac{\lambda}{\mu}(e^{-\mu (\max(u,v,w)-\min(u,v,w))}-e^{-\mu \max(u,v,w)})$$
we can similarly bound $I_{i,j,k}^8$ by a $c\,{n^{-2}}$ for all $i,j,k$.
To summarize, all the $I_{i,j,k}$'s are less than $c\,{n^{-2}}$ for some $c$,
except for the $I^1_{i,i,i}$'s, $i=1,...,n$, which are bounded by a constant but
are only $n$ in number, and all terms where one index appears twice, which are
less than $c\,n^{-1}$ for some $c$, but are only $n^2$ in number. Hence (ii).

\hfill\qed \end{proof}
\begin{proof}[Proof of (iii)]
  We have for all $0\le i \le n-1$,
  \begin{multline}
    \label{eq:normes_u}
    \int\int u^{\sharp}_iu^{\sharp}_i\dnn
    =\int\int\alpha_{i+1}\dnn+\int\int\alpha_{i}\dnn-2
    \int\int\alpha_{i+1}\alpha_{i}\dnn\\+2\mu
    \int\int\beta_i\alpha_{i+1}\dnn-2\mu
    \int\int\beta_i\alpha_{i}\dnn+\mu^2\int\int\beta_i\beta_i\dnn\\=J_1+J_2+J_3+J_4+J_5+J_6,
  \end{multline}
  where straightforward calculations show that
  \begin{align*}
    J_1 &=\frac{\lambda n }{\mu}\left(1-e^{-\mu t^n_{i+1}}\right);\quad J_2=\frac{\lambda n}{\mu}\left(1-e^{-\mu t^n_i}\right);\\
    J_3 &=-2\frac{\lambda n}{\mu}\left(e^{-\frac{\mu T}{n}}-e^{-\mu t^n_{i+1}}\right);\quad J_4 =2\frac{\lambda n}{\mu}(1-e^{-\frac{\mu T}{n}})-2\lambda e^{-\mu t^n_{i+1}};\\
    J_5 &=-2\frac{\lambda n}{\mu}(1-e^{-\frac{\mu T}{n}})-2\frac{\lambda n}{\mu}(e^{-\mu t^n_{i+1}}-e^{-\mu t^n_i});\\
    J_6 &=\lambda\left(2+2e^{-\mu t^{n}_{i+1}}+\frac{2n}{\mu}(e^{-\mu t_{i+1}^{n}}-e^{-\mu t^n_i}+e^{\frac{-\mu T}{n}}-1)\right).
  \end{align*}
  Recalling (\ref{eq:defgamma}), adding up the $J_j$'s, $j=1,...,6$, concludes
  the proof.
\hfill\qed \end{proof}


\end{document}